\documentclass{amsart}
\usepackage{amssymb,latexsym}
\usepackage{enumerate}
\usepackage{a4wide}
\usepackage[all]{xy}

\newtheorem{thm}{Theorem}[section]

\newtheorem{proposition}[thm]{Proposition}

\theoremstyle{definition}

\newtheorem{rem}[thm]{Remark}

\newcommand\m{\mathfrak{m}}
\newcommand\J{\mathcal{J}}
\newcommand\D{\mathcal{D}}
\newcommand\Div{\mathop{\rm Div}\nolimits}
\renewcommand\div{\mathop{\rm div}\nolimits}
\newcommand\Sel{\mathop{\rm Sel}\nolimits}
\newcommand\im{\mathop{\rm im}\nolimits}
\newcommand\Gal{\mathop{\rm Gal}\nolimits}
\newcommand\Norm{\mathop{\rm Norm}\nolimits}
\newcommand\Prin{\mathop{\rm Princ}\nolimits}
\newcommand\Pic{\mathop{\rm Pic}\nolimits}

\renewcommand\H{{\rm H}}
\newcommand\tr{{t}}

\newcommand\Br{\mathop{\rm Br}\nolimits}

\newcommand\Lsep{{L^{\rm s}}}
\newcommand\fake{{\rm fake}}
\newcommand\unfake{{\rm explicit}}

\newcommand\sep{^{\rm s}}
\newcommand\ksep{{k^{\rm s}}}
\newcommand\kvsep{{k_v^{\rm s}}}
\newcommand\Vsep{{V^{\rm s}}}
\newcommand\Csep{C^{\rm s}}

\renewcommand\P{\mathbb{P}}
\newcommand\F{\mathbb{F}}
\newcommand\A{\mathbb{A}}
\newcommand\Z{\mathbb{Z}}
\newcommand\ff{\kappa}
\newcommand\omm{1\, {\rm mod } \,\m}
\newcommand\isom{\cong}
\newcommand\M{M}
\newcommand\Tr{\mathop{\rm Tr}\nolimits}

\begin{document}

\title{Explicit Selmer groups for cyclic covers of~$\P^1$}

\author[M. Stoll]{Michael Stoll}
\address{Mathematisches Institut,
         Universit\"at Bayreuth,
         95440 Bayreuth, Germany.}
\email{Michael.Stoll@uni-bayreuth.de}
\author[R. van Luijk]{Ronald van Luijk}
\address{Mathematisch Instituut,
         Universiteit Leiden,
         Postbus 9512, 2300 RA, Leiden, The Netherlands.}
\email{rvl@math.leidenuniv.nl}

\date{}

\begin{abstract}
For any abelian variety~$J$ over a global field~$k$ and an isogeny $\phi \colon J \to J$, the Selmer group~$\Sel^\phi(J,k)$ is a subgroup of the Galois cohomology group $\H^1(\Gal(\ksep/k), J[\phi])$, defined in terms of local data.
When $J$ is the Jacobian of a cyclic cover of~$\P^1$ of prime degree~$p$, the Selmer group has a quotient by a subgroup of order at most~$p$ that is isomorphic to the `fake Selmer group', whose definition is more amenable to explicit computations.
In this paper we define in the same setting the `explicit Selmer group', which is isomorphic to the Selmer group itself and just as amenable to explicit computations as the fake Selmer group.
This is useful for describing the associated covering spaces explicitly and may
thus help in developing methods for second descents on the Jacobians considered.
\end{abstract}

\maketitle

\section{Introduction}

Let $k$ be a field and $\ksep$ a separable closure of $k$ with Galois group $G_k = \Gal(\ksep/k)$.  Let $C$ be a smooth projective curve over $k$ with Jacobian $J$. Let $\phi\colon J \to J$ be a separable isogeny and $J[\phi]$ the kernel of $\phi \colon J(\ksep) \to J(\ksep)$. Taking Galois invariants of the short exact sequence 
$$
0 \to J[\phi] \to J(\ksep) \xrightarrow{\phi} J(\ksep) \to 0
$$
gives rise to a long exact sequence, which induces another short exact sequence
$$
0 \to J(k)/\phi J(k) \xrightarrow{\delta_\phi} \H^1(G_k, J[\phi]) \to \H^1(G_k,J(\ksep))[\phi] \to 0,
$$
where $\H^1(G_k, J(\ksep))[\phi]$ stands for the kernel of the map $\phi_* \colon \H^1(G_k,J(\ksep)) \to \H^1(G_k, J(\ksep))$ induced by~$\phi$ on cohomology.
If $J(k)$ is finitely generated, which is the case if $k$ is finitely generated as a field over its prime subfield, and if $\phi$ is not an automorphism, then often, including in the cases we will treat, the size of the group $J(k)/\phi J(k)$ yields a bound on the rank of the Mordell-Weil group $J(k)$. As many methods of retrieving arithmetic information about $C$, such as the Mordell-Weil sieve and Chabauty's method, involve the rank of $J(k)$, it is of interest to be able to bound the size of $J(k)/\phi J(k)$, or, equivalently, of its image in $\H^1(G_k, J[\phi])$. Unfortunately, this group $\H^1(G_k, J[\phi])$ is in general very large and hard to handle. 

Now assume that $k$ is a global field. For each place $v$ of $k$, we write $k_v$ for the completion of $k$ at $v$. 
Then the local analogues of the map $J(k)/\phi J(k) \to \H^1(G_k, J[\phi])$ for each place $v$ can be put together to give the following commutative diagram.
$$
\xymatrix{J(k)/\phi J(k) \ar[r]^{\delta_\phi}\ar[d]
            & \H^1(G_k, J[\phi]) \ar[r] \ar[d] \ar[dr]^{\tau}
            & \H^1(G_k, J(\ksep))[\phi] \ar[d] \\
          \prod_v J(k_v)/\phi J(k_v) \ar[r]
            & \prod_v \H^1(\Gal(k_v\sep/k_v), J[\phi]) \ar[r]
            & \prod_v \H^1(G_k, J(\kvsep))[\phi] \\
}
$$
Here $\prod_v$ denotes the product over all places of $k$. 
By definition, the Selmer group $\Sel^\phi(J,k)$ is the kernel of~$\tau$: it consists of all elements of $\H^1(G_k, J[\phi]) $ that map into the image of the local
map $J(k_v)/\phi J(k_v) \to \H^1(\Gal(k_v\sep/k_v), J[\phi]) $ for every $v$. Clearly $\Sel^\phi (J,k)$ contains the image of $\delta_\phi$,
and it can be shown that $\Sel^\phi (J,k)$ is an effectively computable finite group, which  already gives a bound on $J(k)/\phi J(k)$. 
However, the description of $\Sel^\phi (J,k)$ as a subgroup of $\H^1(G_k, J[\phi]) $ is not amenable to explicit computations.

In \cite{PS:1997}, Poonen and Schaefer consider curves $C$ with an affine model given by $y^p = f(x)$, where $p$ is a prime number and $f$ is $p$-power free and splits into linear factors over $\ksep$.  They assume that the characteristic of $k$ is not equal to $p$ and that $k$ contains a primitive $p$-th root $\zeta$ of unity. They take the isogeny to be $\phi = 1-\zeta$, where
$\zeta$ acts on $C$ as $(x,y) \mapsto (x, \zeta y)$. From now on we restrict ourselves to this situation as well.
Note that this includes hyperelliptic curves as the special case
$p=2$; then the isogeny $\phi$ is multiplication by~$2$.
After an automorphism of the $x$-line, we may assume that the map to the $x$-line does not ramify at $\infty$, so that the degree of $f$ is divisible by $p$.\footnote{%
For this to be true, $k$ has to be sufficiently large. Later $k$ will be a global
field, and there will be no problem.}
Let $f_0$ be a radical of $f$, i.e., a separable polynomial in $k[x]$ with the same roots in $\ksep$ as $f$, and set $L = k[T]/f_0(T)$.
We assume that every point in~$J(k)$ can be represented by a $k$-rational divisor on~$C$.
Poonen and Schaefer define a homomorphism $(x-T) \colon J(k) \to L^*/{L^*}^pk^*$ and show that it factors as 
\begin{equation}\label{factorization}
J(k) \to J(k) /\phi J(k) \xrightarrow{\delta_\phi} \Sel^\phi(J,k) \to L^*/{L^*}^pk^*.
\end{equation}
We will recall the definition of this map in Section \ref{xminT}.
For $p=2$ and a polynomial $f$ of degree $4$ with a rational root, the curve $C$ is elliptic; the last map in the factorization is injective in this case and the map $(x-T)$ gives the usual $2$-descent map on $C$. For $p=2$ and $\deg f = 6$,
Cassels~\cite{Cassels:1983} had already defined the map $(x-T)$ (using different notation),
but it was Poonen and Schaefer that related it to the cohomological map $\delta_\phi$ through the given factorization.

In general, and in fact already in Cassels' case, the last map in the factorization need not be injective; its kernel is trivial or isomorphic to $\mu_p$.
Following~\cite{PS:1997}, the image of $\Sel^\phi(J,k)$ in
$L^*/{L^*}^pk^*$ is called the {\em fake Selmer group} $\Sel^\phi_\fake(J,k) $;
it is a quotient of the true Selmer group~$\Sel^\phi(J,k) $.
This means that, although the group~$L^*/{L^*}^pk^*$ is easier to work with
explicitly than~$\Sel^\phi(J,k) $, information may get lost by studying the image of $J(k)/\phi J(k)$ in the former group instead of the latter.

The aim of this paper is to replace the group $L^*/{L^*}^pk^*$ by one that is equally easy to work with and that admits an injection 
from $\Sel^\phi(J,k) $ into it, and thus also from $J(k) / \phi J(k)$. The description of such a group involves a `weighted norm map'~$N$ defined as follows. Let $f = c \prod_j f_j^{m_j}$ be the unique factorization of $f$ over $k$ with $f_j$ monic and $c \in k^*$. For $\beta \in L^*$ we then set
\[ N(\beta) = \prod_j \Norm_{L_j/k}(\beta_j)^{m_j}, \]
where $\beta_j$ is the image of $\beta$ in the field $L_j = k[x]/f_j(x)$.

It turns out that the image of the last map $\Sel^\phi(J,k) \to L^*/{L^*}^pk^*$ of the factorization (\ref{factorization}) is contained in the kernel of the map 
$N \colon L^*/{L^*}^pk^* \to k^*/{k^*}^p$ induced by the weighted norm map. The new group consists of all elements of this kernel, together with some choice of $p$-th root of their norm. More precisely, we will prove the following theorem. 

\begin{thm}\label{mainthm}
Let $k$ be a global field containing a primitive $p$-th root of unity, and let
$C$, $J$, $L$ and~$N$ be as in the discussion above. Assume that for each place $v$ of $k$, the curve $C$ has a $k_v$-rational divisor class of degree~$1$. 
Set $\Gamma = \{ (\delta,n) \in L^* \times k^* \, | \, N(\delta) = n^p\}$ and let $\chi \colon L^* \to \Gamma$ be given by 
$\theta \mapsto (\theta^p, N(\theta))$.  Let $\iota : k^* \to \Gamma$ be defined by $x \mapsto (x,x^{\frac{1}{p}\deg f})$.
Then there is a homomorphism  $(x-T,y) \colon J(k) \to \Gamma/\chi(L^*)\iota(k^*)$ that factors as 
$$
J(k) \to J(k) /\phi J(k) \xrightarrow{\delta_\phi} \Sel^\phi(J,k) \hookrightarrow \Gamma/\chi(L^*)\iota(k^*)
$$
and whose composition with the map $\Gamma/\chi(L^*)\iota(k^*) \to L^*/{L^*}^pk^*$ induced by the projection $\Gamma \to L^*$ equals the map $(x-T)$.  
\end{thm}

The map $(x-T,y)$ will be defined in Section \ref{xminT}. The isomorphic image of $\Sel^\phi(J,k)$ in $ \Gamma/\chi(L^*)\iota(k^*)$ is the
{\em explicit Selmer group} $\Sel^\phi_\unfake(J,k)$.

If all one wants is to get the size of the Selmer group (and thus an upper
bound on the Mordell-Weil rank), then the results of~\cite{PS:1997} are sufficient,
since they tell us exactly the difference between the $\F_p$-dimensions of
$\Sel^\phi(J,k)$ and~$\Sel^\phi_\fake(J,k)$. On the other hand, apart from the
intellectual satisfaction resulting from a nice explicit description of the Selmer
group itself, the additional
information given by identifying $\Sel^\phi(J,k)$ with the explicit Selmer
group gives us a handle on the covering spaces corresponding to its elements:
in~\cite{FLT:2011} equations for the covering spaces are given in the genus two
case that depend on the image of the Selmer group element in the fake Selmer
group together with a square root of its norm, which is precisely the information
contained in the corresponding element of the explicit Selmer group. Explicit
models of these covering spaces are useful for the search of potentially large
Mordell-Weil generators and can also serve as a starting point for second descents.
In particular, one can hope that our explicit Selmer group can be used to
extend Cassels' method for computing the Cassels-Tate
pairing on the 2-Selmer group of an elliptic curve~\cite{Cassels:1998}, which
uses the quadratic Hilbert symbol on elements of the explicit version of the
Selmer group, to Jacobians of curves of genus~two.

Since our results here extend and improve what Poonen and Schaefer have done,
much of this paper is based on Poonen and Schaefer's paper~\cite{PS:1997}, including the weighted norm map~$N$. The main new element brought in is the group~$\Gamma$ of Theorem~\ref{mainthm}, which was first introduced in~\cite{FLT:2011}.
The recent preprint~\cite{BPS:2012} contains in its appendix a general recipe for
turning `fake' Selmer groups into `explicit' ones, which was developed as a
generalization of the method given in~\cite{SchSt:2004} for $p$-descent on
elliptic curves with $p$~odd and of the approach described here. 
Our result could in principle also be obtained as a special case of this general recipe.
However, the more direct approach used here leads to a much simpler proof.

In the next section we will introduce some notation, all following \cite{PS:1997}. In Section~\ref{cohomexp} we identify some cohomology groups with more explicit groups such as those mentioned in Theorem \ref{mainthm}. In Section \ref{xminT} we define the maps $(x-T)$ and $(x-T,y)$, so that in the last section we can `unfake' the fake Selmer group and replace it with the explicit Selmer group by proving Theorem~\ref{mainthm}.


\section{Notation}

Our setting will be the same as in \cite{PS:1997}. Let $p$ be a prime.  Let $k$ be a field of characteristic not equal to  $p$ and let $\ksep$ be a separable closure of $k$ with Galois group $G_k = \Gal(\ksep/k)$.  
Assume that $k$ contains a primitive $p$-th root of unity.  
For any $G_k$-module $A$ and any integer $i \geq 0$ 
we abbreviate the cohomology group 
$\H^i(G_k,A)$ by $\H^i(A)$. Let $\pi \colon C \to \P^1$ be a cyclic cover of $\P^1$ over $k$ of degree $p$ such that all branch points are in $\P^1(\ksep)\setminus\{\infty\}$. By Kummer theory, the curve $C$ has a (possibly singular) model in $\A^2(x,y)$ given by $y^p=f(x)$, where $f \in k[x]$ factors over $\ksep$ as 
$$
f(x) = c \prod_{\omega\in \Omega} (x - \omega)^{a_\omega}
$$
with $c \in k^*$, with $1\leq a_\omega < p$ for all $\omega$ in the set $\Omega \subset \ksep$ of roots of $f$, and where $p$ divides the degree $\deg f = \sum_\omega a_\omega$ of $f$. Set $d = \# \Omega$. By the Riemann--Hurwitz formula the genus of $C$ equals $g(C) = (d-2)(p-1)/2$. 

For any $k$-variety $V$, we write $\Vsep = V\times_k \ksep$, while $\ff(V)$ and $\ff(\Vsep)$ denote the function fields of $V$ and $\Vsep$. Let $\Div \Csep$ be the group of all divisors on $\Csep$. If $f \in \ff(\Csep)^*$, we denote the divisor
of~$f$ by $\div(f) \in \Div \Csep$. We let
$\Prin \Csep = \{\div(f) : f \in \ff(\Csep)^*\}$ be the subgroup of principal divisors. Set $\Pic \Csep = \Div \Csep/ \Prin \Csep$. Also set 
\begin{align*}
\Div C &= \H^0(\Div \Csep)\\
\Prin C &= \H^0(\Prin \Csep), \\
\Pic C &= \Div C / \Prin C.
\end{align*}
As in \cite{PS:1997}, we consider the divisor $\m = \pi^* \infty \in \Div C$, the sum of all $p$ points above $\infty \in \P^1$.  For any function $h$ in the function field $\ff(\Csep)$ of $\Csep$ we say that $h$ is $\omm$ if $h(P)=1$ for all points $P$ in the support of $\m$ (for a more general definition, see \cite[section 2]{PS:1997}).
Let $\Div_\m \Csep \subset \Div \Csep$ be the group of all divisors with support disjoint from $\m$, and let $\Prin_\m \Csep \subset \Prin \Csep$ be the subgroup of all principal divisors of functions that are $\omm$.
Set $\Pic_\m \Csep = \Div_\m \Csep / \Prin_\m \Csep$ and
\begin{align*}
\Div_\m C &= \H^0(\Div_\m \Csep)\\
\Prin_\m C &= \H^0(\Prin_\m \Csep), \\
\Pic_\m C &= \Div_\m C / \Prin_\m C.
\end{align*}
Let $\Div^0 \Csep \subset \Div \Csep$ be the subgroup of divisors of degree $0$ and let $\Div^0 C$, $\Pic_\m^0 \Csep$, etc. be the degree-zero parts of the corresponding groups.
Let $\Div^{(p)} \Csep \subset \Div \Csep$ be the subgroup of divisors of degree divisible by $p$ and let $\Div^{(p)}C$, $\Pic_\m^{(p)} \Csep$, etc. be the degree-divisible-by-$p$ parts of the corresponding groups. Let $J$ and $J_\m$ denote the Jacobian of $C$ and the generalized Jacobian of the pair $(C,\m)$, respectively, so that $J(\ksep) = \Pic^0 \Csep$ and  $J_\m(\ksep) = \Pic_\m^0 \Csep$. We write $J[p]$ and $J_\m[p]$ for the kernel of multiplication-by-$p$, written as $[p]$, on $J(\ksep)$ and $J_\m(\ksep)$, respectively.
We denote the trivial group in diagrams by 1.


\section{Making cohomology groups explicit}\label{cohomexp}

Pick any $c_0 \in k^*$ and define a radical $f_0=c_0\prod_{\omega \in \Omega} (x-\omega)\in k[x]$ of $f$.  Set $L = k[X]/f_0(X)$ and $\Lsep = L \otimes_k \ksep$. We will denote the image of $X$ in $L$ and $\Lsep$ by $T$.
By the Chinese Remainder Theorem, the $\ksep$-linear maps $\rho_\omega \colon \Lsep \to \ksep, T \mapsto \omega$ combine to an isomorphism 
$$
\rho = (\rho_\omega)_{\omega\in \Omega} \colon \Lsep \to \prod_{\omega\in \Omega} \ksep,
$$
which restricts to the diagonal embedding on $\ksep \subset \Lsep$. 
From now on, whenever $\omega$ is used as index, it ranges over all elements of $\Omega$.
Note that the induced Galois action on 
$\prod_\omega \ksep$ is given by acting on the indices as well, so by $\sigma\big( (a_\omega)_\omega\big) = \big(\sigma(a_{\sigma^{-1}\omega})\big)_\omega$. 
We often identify $\Lsep$ with $\prod_\omega \ksep$ through $\rho$, thereby identifying $T$ with the element $(\omega)_\omega$.
For any commutative ring $R$, we let $\mu_p(R)$ denote the kernel of the homomorphism $R^* \to R^*, x \mapsto x^p$. We abbreviate $\mu_p(k) = \mu_p(\ksep)$ by $\mu_p$ and note that $\rho$ induces an isomorphism $\mu_p(\Lsep) \to \prod_\omega \mu_p$.  Let the `weighted norm map'
$N\colon \Lsep \isom \prod_\omega \ksep \to \ksep$ be given by 
$(\beta_\omega)_\omega \mapsto \prod_\omega \beta_\omega^{a_\omega}$.
Since $p$ divides $\sum_\omega a_\omega$, the kernel of $N$ contains $\mu_p$. 
The map $N$ is Galois-equivariant, as for conjugate roots $\omega, \omega' \in \Omega$ we have $a_\omega = a_{\omega'}$, so it induces a map $N \colon L \to k$.  
This map is the same as the norm map $N$ that was defined in the introduction.
Let $\M$ denote the kernel of the induced 
map $N \colon \mu_p(\Lsep) \to \mu_p$.  Then we obtain the following commutative diagram, in which the horizontal and vertical sequences are exact. 
\begin{equation}\label{defM}
\xymatrix{
&1\ar[d]&1\ar[d]\cr
&\mu_p \ar[d] \ar@{=}[r]& \mu_p \ar[d] \cr
1\ar[r] & \M \ar[r]\ar[d] &\mu_p(\Lsep) \ar[d]\ar[r]^(0.58){N}&\mu_p\ar[r]
\ar@{=}[d]&1\cr
1 \ar[r] & \M/\mu_p \ar[r] \ar[d]&\mu_p(\Lsep)/\mu_p
\ar[r]^(0.63)N \ar[d]& \mu_p \ar[r] & 1\cr
&1&1 \cr
}
\end{equation}
Note that the map $N \colon \mu_p(\Lsep) \to \mu_p$ is surjective because 
we can take $1$ in each component of $\mu_p(\Lsep) \isom \prod_\omega \ksep$ except for one component, say corresponding to $\omega$, where we choose an $a_\omega$-th root of $\zeta$, which exists because the greatest common divisor $(a_\omega,p)$ equals $1$.

We will give a concrete description of the Galois cohomology groups $\H^1(M)$ and $\H^1(\mu_p(\Lsep))$ and their images in 
$\H^1(M/\mu_p)$ and $\H^1(\mu_p(\Lsep)/\mu_p)$. Let $\partial \colon L^* \times k^* \to k^*$ be the homomorphism that sends $(\delta,n)$ to $N(\delta) n^{-p}$ and let $\partial\sep$ denote the corresponding map from ${\Lsep}^* \times {\ksep}^*$ to ${\ksep}^*$. Set 
$$
\Gamma\sep = \ker \partial\sep
= \left\{(\delta,n) \in {\Lsep}^*\times {\ksep}^*\,\,|\,\, N(\delta) = n^p\right\}
$$
and 
$$
\Gamma = \H^0(\Gamma\sep) = \ker \partial
= \left\{(\delta,n) \in {L}^*\times {k}^*\,\,|\,\, N(\delta) = n^p\right\}.
$$
We will write $\iota$ for the injection $\ksep^* \to \Gamma\sep$
given by $x \mapsto (x,x^{\frac{1}{p}\deg f})$; it restricts to an injection
$\iota : k^* \to \Gamma$.
Let the map $\chi\colon \Lsep^* \to \Gamma\sep$ be given by $\theta \mapsto  (\theta^p,N(\theta))$. It is surjective, has kernel $M$, and it  restricts to 
a map $\chi\colon L^* \to \Gamma$.  The long exact sequence associated to the short exact sequence
\begin{equation}\label{alsoM}
1 \to M \to \Lsep^* \xrightarrow{\chi} \Gamma\sep \to 1
\end{equation}
contains the connecting map $\delta_\chi \colon \Gamma \to \H^1(M)$, which sends $(\delta,n)$ to the class of the cocycle $G_k \ni \sigma \mapsto \sigma(\theta)/\theta \in M$ for a fixed choice of $\theta \in \Lsep^*$ with $\chi(\theta) = (\delta,n)$.  Similarly, the short exact sequence
\begin{equation}\label{mupL}
1 \to \mu_p(\Lsep) \to \Lsep^* \xrightarrow{x \mapsto x^p} \Lsep^* \to 1
\end{equation}
provides a connecting map $\delta_p \colon L^* \to \H^1(\mu_p(\Lsep))$.
Parts of the following proposition were proved for $p=2$ in~\cite[Proposition 2.6]{FLT:2011}.

\begin{proposition}\label{diagramone}
The map $\delta_\chi$ induces an isomorphism $\delta_\chi \colon \Gamma/\chi(L^*) \to \H^1(M)$ and an isomorphism from $\Gamma/\chi(L^*)\iota(k^*)$ to the image of  $\H^1(M)$ in $\H^1(M/\mu_p)$.
The map $\delta_p$ induces an isomorphism $\delta_p \colon L^*/{L^*}^p \to \H^1(\mu_p(\Lsep))$ and an isomorphism from $L^*/{L^*}^pk^*$ to the image of  $\H^1(\mu_p(\Lsep))$ in $\H^1(\mu_p(\Lsep)/\mu_p)$. 
These maps fit in the commutative diagram 
$$
\xymatrix{
\mu_p \ar[rr]\ar@{=}[dd]\ar[rd]&& \H^1(M) \ar'[d][dd] \ar[rr] && \H^1(\mu_p(\Lsep)) \ar[dd]\\
&\Gamma/\chi(L^*) \ar[ru]^{\delta_\chi}\ar[rr] \ar[dd] && L^*/{L^*}^p \ar[dd] \ar[ru]_{\delta_p} \\
\mu_p\ar'[r][rr]\ar[rd]&& \H^1(M/\mu_p) \ar'[r][rr] && \H^1(\mu_p(\Lsep)/\mu_p) \\
&\Gamma/\chi(L^*)\iota(k^*) \ar[ru]^{\delta_\chi} \ar[rr] && L^*/{L^*}^pk^* \ar[ru]_{\delta_p}
}
$$
where the back face consists of part of the long exact sequences associated to the horizontal sequences in (\ref{defM}), the vertical maps in the front face are the obvious quotient-by-$k^*$ maps, the horizontal maps in the front face are induced by the projection map $\Gamma \to L^*, \, (\delta,n) \mapsto \delta$, and the remaining maps from $\mu_p$ send $\zeta \in \mu_p$ to the class of $(1,\zeta)$. 
\end{proposition}
\begin{proof}
The commutativity of the front and back face are obvious. The projection map 
$\Gamma\sep \to \Lsep^*,\,\, (\delta,n) \mapsto \delta$ induces a map between the short exact sequences (\ref{alsoM}) and (\ref{mupL}). Part of the associated long exact sequences gives the following diagram.
$$
\xymatrix{
L^* \ar[r]^\chi \ar@{=}[d]& \Gamma \ar[d]\ar[r]^(0.45){\delta_\chi} & \H^1(M)\ar[d] \ar[r]  & \H^1(\Lsep^*)\ar@{=}[d] \\
L^* \ar[r]^{x \mapsto x^p} & L^* \ar[r]^(0.35){\delta_p} & \H^1(\mu_p(\Lsep)) \ar[r] & \H^1(\Lsep^*)
}
$$
By a generalization of Hilbert's Theorem~90 the group  
$\H^1(\Lsep^*)$ is trivial (see~\cite[Exercise~X.1.2]{locflds:1979}). The commutativity of the quadrilateral in the top face follows, as well as the fact that the maps $\delta_\chi$ and $\delta_p$ in it are isomorphisms.
Similarly, also using that $\H^1(\ksep^*)$ vanishes by Hilbert's Theorem~90, the natural maps from the short exact sequence
\begin{equation}\label{mupk}
1 \to \mu_p \to \ksep^* \xrightarrow{x \mapsto x^p} \ksep^* \to 1
\end{equation}
to (\ref{alsoM}) and (\ref{mupL}) yield long exact sequences that induce the following diagrams.
$$
\xymatrix{
k^*/{k^*}^p \ar[d]\ar[r]^\isom & \H^1(\mu_p) \ar[d] && k^*/{k^*}^p \ar[d]\ar[r]^\isom & \H^1(\mu_p) \ar[d]\\
\Gamma/\chi(L^*) \ar[r]^\isom_{\delta_\chi} & \H^1(M) && L^*/{L^*}^p \ar[r]^(0.45){\isom}_(0.45){\delta_p} & \H^1(\mu_p(\Lsep))
}
$$
The associated maps on cokernels of the vertical homomorphisms induce the claimed isomorphisms from $L^*/{L^*}^pk^*$ to the image of  $\H^1(\mu_p(\Lsep))$ in $\H^1(\mu_p(\Lsep)/\mu_p)$ and from $\Gamma/\chi(L^*)\iota(k^*)$ to the image of  $\H^1(M)$ in $\H^1(M/\mu_p)$. This also implies the commutativity of the left and right faces of the cube in the diagram. Commutativity of the quadrilateral in the bottom face follows immediately from the commutativity of the other faces of the cube and the fact that the quotient map $\Gamma/\chi(L^*) \to \Gamma/\chi(L^*)\iota(k^*)$ is surjective. Finally, choose a $\theta \in \mu_p(\Lsep)$ with $N(\theta)=\zeta$.  Then the image of $\zeta$ in $\H^1(M)$ is represented by the cocycle $\sigma \mapsto \sigma(\theta)/\theta$, which coincides with $\delta_\chi((1,\zeta))$. It follows that also the triangular prism in the diagram commutes. 
\end{proof}


\section{A new map}\label{xminT}

Let $h$ be a nonzero rational function on~$C$. Then we can extend evaluation of~$h$
on points not in the support of~$\div(h)$ multiplicatively to divisors whose
support is disjoint from that of~$\div(h)$ by setting
\[ h(D) = \prod_P h(P)^{n_P} \qquad\text{if $D = \sum_P n_P P$.} \]
If $K$ is a field extension of~$k$ that is a field of definition of~$h$,
then this defines a group homomorphism from the group of $K$-defined divisors
with support disjoint from that of~$\div(h)$ into the multiplicative group
of~$K$.

In the following, we will frequently work with objects defined over~$L$.
There are (at least) two ways of interpreting what these objects mean.
We can either just think of them as $L$-defined objects (functions, points, etc.),
allowing \'etale algebras over~$k$ instead of only field extensions. Or else
we remind ourselves that the elements of~$L$ correspond to Galois-equivariant
maps from~$\Omega$ into~$\ksep$; then a function defined over~$L$ can be
considered as a Galois-equivariant map from~$\Omega$ into $\ff(\Csep)$, etc.
Sometimes, we use $\Lsep$ in place of~$L$; then the corresponding maps
from~$\Omega$ need not be Galois-equivariant. In this sense, $\mu_p(\Lsep)$
denotes the set of maps $\Omega \to \mu_p$, and $M$ denotes the subset
of maps~$\eta$ such that $N(\eta) = \prod_\omega \eta(\omega)^{a_\omega} = 1$.

For example, we let $W = (T, 0) \in C(L)$ be a `generic ramification point'
on~$C$. In the second interpretation, $W$ corresponds to the map
$\omega \mapsto (\omega, 0)$ that gives all the ramification points on~$C$
indexed by the roots of~$f$. In this section, we will consider the
function $x-T$, which is an $L$-defined rational function on~$C$. In our
second interpretation, we associate to each $\omega \in \Omega$ the
rational function $x-\omega \in \ff(\Csep)$. We have
\[ \div(x - T) = p W - \m \qquad\text{and}\qquad
   \div(y) = \Tr W - \tfrac{1}{p}(\deg f)\,\m \,,
\]
where $\Tr W = \sum_\omega a_\omega (\omega,0)$ denotes the `trace' of~$W$,
the additive analogue of the weighted norm~$N$. In other words, in our first
interpretation $W$ is a prime divisor in $\Div C_L$, while in the second 
interpretation it corresponds with a Galois-equivariant map 
$\Omega \to \Div \Csep$ sending $\omega$ to the prime divisor $(\omega,0)$, 
the images of which have weighted sum $\Tr W \in \Div C$. 

A divisor  on $C\sep$ is called {\em good} if its support is disjoint from
$\m = \pi^* \infty$ and the ramification points of $\pi$, i.e., disjoint 
from the support of $\div(y)$. This also means that the
support is disjoint from the support of~$\div(x-T)$.
Let $\Div_\perp \Csep$ denote the group of good divisors on $\Csep$, and set 
$\Div_\perp C = \H^0(\Div_\perp \Csep)$.  Every divisor class in $\Pic \Csep$ and $\Pic_\m \Csep$ 
is represented by a good divisor. Let $\Div_\perp^0 \Csep$, $\Div_\perp^0 C$, 
$\Div_\perp^{(p)} \Csep$, and $\Div_\perp^{(p)} C$ denote the obvious groups. 
By the introductory remarks of this section, the function $x-T$ defines
homomorphisms
\[ (x-T) \colon \Div_\perp C \to L^* \qquad\text{and}\qquad
   (x-T) \colon \Div_\perp \Csep \to \Lsep^* \,.
\]
We define the map 
$$
\alpha \colon J_\m[p] \to {\Lsep}^*, \qquad \D \mapsto \frac{(x-T)(D)}{h(W)},
$$
where $D$ is a good divisor representing the class $\D$, and where $h \in \ff(\Csep)$ is the unique function that is $\omm$ and satisfies $\div(h) = pD$. As before, $h(W)$
can be interpreted as the map $\omega \mapsto h((\omega,0)) \in \ksep^*$.
Note that $\alpha$ is well-defined as for any representative $D'$ of $\D$ there is a function $g \in \ff(\Csep)$ that is $\omm$ with $\div(g)=D'-D$, so that
$\div(g^ph) = pD'$; by Weil reciprocity we have 
\begin{align*}
  \frac{(x-T)(D')}{(g^ph)(W)}
    &= \frac{(x-T)(\div(g)+D)}{g^p(W)h(W)}
     = \frac{g(\div(x-T))}{g(pW)}\cdot\frac{(x-T)(D)}{h(W)} \\
    &= \frac{g(pW - \m)}{g(pW)} \cdot \frac{(x-T)(D)}{h(W)}
     = \frac{g(pW) g(\m)^{-1}}{g(pW)} \cdot \frac{(x-T)(D)}{h(W)}
     = \frac{(x-T)(D)}{h(W)},
\end{align*}
since $g(\m) = 1$.
We will see that $\alpha$ induces an isomorphism between $M$ and the kernel of an endomorphism of $J_\m$ that we now define.

The group $\mu_p$ acts on $C$ and $\Csep$ by letting $\zeta \in \mu_p$ act as $(x,y) \mapsto (x,\zeta y)$.
Linear extension gives a Galois-equivariant action on $\Div \Csep$ by the group ring $\Z[\mu_p]$. The element $\tr = \sum_{\zeta \in \mu_p} \zeta \in \Z[\mu_p]$
sends a point $Q \in \Csep(\ksep)$ to the divisor $\tr(Q) = \pi^*(\pi Q)$, which is linearly equivalent to $\m$. We conclude that $\tr$ sends a divisor $D \in \Div \Csep$ to a divisor linearly equivalent to $(\deg D) \m$, and the subgroups $\Div^0 \Csep$ and $\Div_\m^0 \Csep$ to $\Prin \Csep$ and $\Prin_\m \Csep$, respectively. This implies that the induced action of $\Z[\mu_p]$ on $J$, on $J_\m$, on 
$\Pic^0 C$, and on $\Pic_\m^0 C$ factors through the quotient $\Z[\mu_p]/\tr$, which is isomorphic to the cyclotomic subring of $k$ generated by $\mu_p$. 

Fix, once and for all, a primitive $p$-th root of unity $\zeta \in \mu_p$, so that this cyclotomic ring is equal to $\Z[\zeta]$. Set
$$
\phi = 1 - \zeta \, \qquad \text{and} \qquad \psi = -\sum_{i=1}^{p-1} i \zeta^i
$$
and notice that $\phi\psi = p$. 
Note that this is slightly different from \cite{PS:1997}, where $\phi$ and $\psi$ are defined as elements of the group ring $\Z[\mu_p]$. 
Let $J_\m[\phi]$ and $J[\phi]$ denote the kernels of the action of $\phi$ on 
$J_\m(\ksep)$ and $J(\ksep)$ respectively.

\begin{proposition}\label{epsilon}
There is an isomorphism $\epsilon \colon J_\m[\phi] \to M$ such that the homomorphism $\alpha$ is
the composition of $\psi \colon J_\m[p] \to J_\m[\phi]$ and $\epsilon$.
Furthermore, $\epsilon$ induces an isomorphism $J[\phi] \to M/\mu_p$.
\end{proposition}
\begin{proof}
This is extracted from~\cite{PS:1997}.
Let $\J_\m[p]$ denote the $p$-torsion of the group $\Pic_\m \Csep / \langle \m' \rangle$, where $\m'$ denotes the class of $\pi^* P$ for any $P \in \A^1(k) \subset \P^1(k)$. By \cite[Section~7]{PS:1997} there is a pairing
$$
e_p \colon \J_\m[p] \times \J_\m[p] \to \mu_p,
$$
defined for a pair $(\D_1,\D_2)$ of classes, represented respectively by divisors $D_1$ and $D_2$ with disjoint support, to be 
$$
e_p(\D_1,\D_2) = (-1)^{d_1d_2} \frac{h_2(D_1)}{h_1(D_2)},
$$
where for $i=1,2$ we have $d_i = \deg D_i$, while $h_i \in \ff(\Csep)$ is the unique function such that 
$x^{-d_i}h_i$ is $\omm$ and $\div(h_i) = pD_i - d_i \m$. 
Note that the group $J_\m[p] \isom \Pic_\m^0(\Csep)[p]$ is a subgroup of $\J_\m[p]$. 
By \cite[Section~6 and Prop.~7.1]{PS:1997} there is an isomorphism $\epsilon \colon J_\m[\phi] \to M$ such that
$\epsilon (\psi \D) = e_p(\D,W)$ for all $\D \in J_\m[p]$. 
Let the class $\D \in J_\m[p]$ be represented by a good divisor $D$, automatically of degree $d_1=0$, and let $h \equiv \omm$ be a function satisfying $\div(h)  = pD$. Note that 
$x^{-1}(x-T)$ is $\omm$, so that we can take $x-T$ as the function corresponding
to~$W$ in the definition of~$e_p$. Therefore, we have 
$$
\epsilon(\psi \D) = e_p(\D,W) = (-1)^0 \frac{(x-T)(D)}{h(W)} = \alpha(\D),
$$
which shows that $\alpha$ factors as claimed. For the fact that $\epsilon$ induces an isomorphism 
$J[\phi] \to M/\mu_p$, see \cite[Section~6]{PS:1997}.
\end{proof}

As in \cite{PS:1997}, we denote the isomorphisms $J_\m[\phi] \to M$ and $J[\phi] \to M/\mu_p$ from Proposition \ref{epsilon} both by $\epsilon$. 

Next, we define the homomorphism 
$$
(\gamma y) \colon \Div_\perp^{(p)} \Csep \to {\ksep}^*, \qquad \sum_P n_P (P) \mapsto c^{-\frac{1}{p}\sum n_P}  \prod_P y(P)^{n_P},
$$
where $c$ is the leading coefficient of $f$ as before.  
This map descends to a map $(\gamma y) \colon \Div_\perp^{(p)} C\to k^*$. 
The name $(\gamma y)$ comes from the fact that if we choose any $p$-th root $\gamma \in \ksep$ of $c^{-1}$, then the map $(\gamma y)$  is the restriction to $\Div_\perp^{(p)} \Csep$ of evaluation of $\gamma y$ on $\Div_\perp \Csep$.  On $\Div_\perp^0 \Csep$ it is also induced by evaluation of $y$. Therefore, when appropriate, we may refer to the map $(\gamma y)$ as just $y$. We remark that
\begin{equation} \label{eqn:relx-Ty}
  N(x-T) = \prod_\omega (x-\omega)^{a_\omega}
         = c^{-1} f(x) = c^{-1} y^p = (\gamma y)^p \,.
\end{equation}

Our main result gives a cohomological interpretation of the
combined map
\[ (x-T, \gamma y) \colon \Div_\perp^{(p)} \Csep \to \Lsep^* \times {\ksep}^* \,. \]
To this end, let $\epsilon_*$ denote the maps on cohomology induced by both maps $\epsilon$. The short exact sequences 
$$
1 \to J_\m[\phi] \to J_\m(\ksep) \xrightarrow{\phi} J_\m(\ksep) \to 1
  \qquad\text{and}\qquad
1 \to J[\phi] \to J(\ksep) \xrightarrow{\phi} J(\ksep) \to 1
$$
induce connecting maps $J_\m(k) \to \H^1(J_\m[\phi])$ and $J(k) \to \H^1(J[\phi])$ that we both denote by $\delta_\phi$.

\begin{thm}\label{main}
The map 
$$
(x-T,\gamma y) \colon \Div_\perp^{(p)} \Csep \to \Lsep^* \times {\ksep}^*, \quad 
D \mapsto \big((x-T)(D), (\gamma y)(D)\big)
$$
induces natural homomorphisms $\Pic^0_\m C \to \Gamma/\chi(L^*)$ and 
$\Pic^0 C \to \Gamma/\chi(L^*)\iota(k^*)$ making the following diagram commutative.
$$
\xymatrix@!C=0.5in{
& J_\m(k) \ar'[d][dd] \ar[rr]^{\delta_\phi} && \H^1(J_\m[\phi]) \ar[rr]^(0.35){\epsilon_*}_(0.35){\isom} \ar'[d][dd] &&  \H^1(M) \ar[dd]\\
\Pic_\m^0 C\ar[dd]\ar[ru]^{\isom} \ar[rrrr]^{(x-T,\gamma y)} &&&& \Gamma/\chi(L^*) \ar[ru]_{\delta_\chi}^(0.45){\isom}\ar[dd]\\
& J(k) \ar[rr]^{\delta_\phi}  && \H^1(J[\phi]) \ar'[r][rr]^(0.35){\epsilon_*}_(0.35){\isom} &&  \H^1(M/\mu_p) \\
\Pic^0 C \ar@{^{(}->}[ru] \ar[rrrr]^{(x-T,\gamma y)} &&&& \Gamma/\chi(L^*)\iota(k^*) \ar@{^{(}->}[ru]_{\delta_\chi}
}
$$
\end{thm}
\begin{proof}
For any good divisor $D=\sum_P n_P(P)$ of degree divisible by $p$ we have,
using~\eqref{eqn:relx-Ty},
\[ N\bigl((x-T)(D)\bigr) = \bigl(N(x-T)\bigr)(D) = (c^{-1} y^p)(D) = (\gamma y)(D)^p, \]
so $(x-T,\gamma y)$ induces a homomorphism $\Div_\perp^{(p)} C \to \Gamma$. Suppose $D\in \Div_\perp^{0} C$ is principal, say 
$D = \div(h)$ for some $h \in \ff(C)^*$. Then by Weil reciprocity we have 
$$
(x-T)(D) = (x-T)\bigl(\div(h)\bigr)
         = h\bigl(\div(x-T)\bigr) = h(pW - \m) = h(W)^p \cdot h(\m)^{-1}
$$
and 
\[ (\gamma y)(D) = y\bigl(\div(h)\bigr) = h\bigl(\div(y)\bigr)
      = h\left(\Tr W - \tfrac{1}{p}(\deg f) \m\right)
      = N(h(W)) \cdot h(\m)^{-\frac{1}{p}\deg f}.
\]
We therefore find 
$$
(x-T,\gamma y)(D) = \chi(h(W))\cdot \iota(h(\m)^{-1}).
$$
This is contained in $\chi(L^*)\iota(k^*)$ and if $h$ is $\omm$ then in fact in $\chi(L^*)$. As every class in $\Pic^0 C$ and $\Pic_\m^0 C$ is represented by a good divisor, we obtain the claimed homomorphisms and see that the front face of the diagram commutes.

The commutativity of the right-side face follows from Proposition \ref{diagramone}, while that of the back and left-side faces is obvious. For the top face, take any $\D \in \Pic_\m^0 C$, represented by a good divisor $D\in \Div_\perp^0 C$, and choose a class $\D' \in \Pic_\m^0 \Csep \isom J_\m(\ksep)$ with $p\D' = \D$ and a good divisor $D'\in \Div^0_\perp \Csep$ representing $\D'$. Then $\phi(\psi \D') = p\D' = \D$, so $\delta_\phi(\D)$ is represented by the cocycle that sends $\sigma \in G_k$ to $\sigma(\psi \D') - \psi\D' = \psi(\sigma(\D')-\D')$ and $\epsilon_*(\delta_\phi(\D))$ is represented by $\sigma \mapsto \epsilon(\psi(\sigma(\D')-\D'))$.
Let $h$ be a function that is $\omm$, satisfying
\[ \div(h) = pD'-D\,, \]
so that $\div(\sigma(h)/h) = p(\sigma(D') - D')$.
Therefore, by Proposition \ref{epsilon}, the class $\epsilon_*(\delta_\phi(\D))$ is represented by the cocycle that sends $\sigma$ to  
$$
  \epsilon\left(\psi(\sigma(\D')-\D')\right)
    = \alpha(\sigma(\D')-\D')
    = \frac{(x-T)(\sigma(D')-D')}{(\sigma(h)/h)(W)}
    = \frac{\sigma(\theta)}{\theta},
$$
for all $\sigma \in G_k$, with 
$$
\theta = \frac{(x-T)(D')}{h(W)}.
$$
We now show that $\chi(\theta) = (\theta^p, N(\theta))$ equals $(x-T,\gamma y)(D)$.
In the first component, we have
\[ \theta^p
     = \frac{(x-T)(D')^p}{h(W)^p}
     = \frac{(x-T)(pD')}{h(pW)}
     = \frac{(x-T)(\div(h) + D)}{h(\div(x-T) + \frac{1}{p}(\deg f)\m)}
     = (x-T)(D)
\]
by Weil reciprocity and the fact that $h(\m) = 1$.
In the second component, we similarly have
\[ N(\theta)
     = \frac{N\bigl((x-T)(D')\bigr)}{N(h(W))}
     = \frac{y(D')^p}{h(\Tr W)}
     = \frac{y(pD')}{h(\div(y) + \tfrac{1}{p}(\deg f)\m)}
     = \frac{y(\div(h) + D)}{h(\div(y) + \tfrac{1}{p}(\deg f)\m)}
     = y(D)\,.
\]
This implies that $\delta_\chi \left((x-T,\gamma y)(\D)\right)$ is represented by the cocycle $\sigma \mapsto \sigma(\theta)/\theta$ as well, so the top face of the diagram commutes indeed.
Finally, commutativity of the bottom face of the diagram follows from commutativity of the other faces and the fact that the map 
$\Pic_\m^0 C \to \Pic^0 C$ is surjective.
\end{proof}

The diagrams of  Proposition \ref{diagramone} and Theorem \ref{main} combine to the following diagram.
\begin{equation}\label{together}
\xymatrix@!C=0.5in{
& J_\m(k) \ar'[d][dd]|-(0.31){\rule{0mm}{1.8mm}}  \ar[rr]^(0.35){\delta_\phi} && \H^1(J_\m[\phi]) \isom  \H^1(M) \ar'[d][dd]|-(0.31){\rule{0mm}{1.8mm}}   \ar[rr] && \H^1(\mu_p(\Lsep)) \ar[dd]\\
\Pic_\m^0 C\ar[dd]\ar[ru]^{\isom} \ar[rr]^(0.36){(x-T, y)} \ar@/_5.5mm/@{-->}[rrrr]_(0.35){(x-T)}&& \Gamma/\chi(L^*) \ar[ru]_{\delta_\chi}^(0.45){\isom}\ar[rr] \ar[dd]|-(0.12){\rule{0mm}{1.8mm}} && L^*/{L^*}^p \ar[dd] \ar[ru]_{\delta_p} ^{\isom}\\
& J(k) \ar'[r]_(0.65){\delta_\phi}[rr]  && \H^1(J[\phi]) \isom  \H^1(M/\mu_p)  \ar'[r][rr] && \H^1(\mu_p(\Lsep)/\mu_p) \\
\Pic^0 C \ar@{^{(}->}[ru] \ar[rr]^{(x-T, y)} \ar@/_5.5mm/@{-->}[rrrr]_(0.35){(x-T)}&& \Gamma/\chi(L^*)\iota(k^*) \ar@{^{(}->}[ru]_{\delta_\chi}\ar[rr] && L^*/{L^*}^pk^* \ar@{^{(}->}[ru]_{\delta_p}\\
}
\end{equation}
The two compositions of the horizontal maps in the front face of this diagram, indicated by dashed arrows, are the $(x-T)$ maps that play a major role in \cite{PS:1997}. Indeed, if we replace the front face by the diagram
$$
\xymatrix{
\Pic_\m^0 C \ar[d] \ar[rrr]^(0.45){(x-T)} &&& L^*/{L^*}^p \ar[d]\\
\Pic^0 C \ar[rrr]^(0.45){(x-T)} &&& L^*/{L^*}^pk^*
}
$$
then all information in this restricted diagram can already be found in \cite{PS:1997}.

\begin{rem}
As explained in \cite[Section 10]{PS:1997}, the group $\Pic^0 C$ is the largest subgroup of $J(k)$ whose image under the map 
$J(k) \to \H^1(\mu_p(\Lsep)/\mu_p)$ is contained in the image of $L^*/{L^*}^pk^*$.  Similarly, it is the largest subgroup whose image under $J(k) \to \H^1(J[\phi])$ is contained in the image of $\Gamma/\chi(L^*)\iota(k^*)$. 
\end{rem}


\section{`Unfaking' the fake Selmer group}

In this section, we make the additional assumption that $k$ is a global  field. For each place $v$ of $k$, we let $k_v$ denote the completion at $v$, with absolute Galois group $G_v = \Gal(k_v\sep/k_v)$; we set $L_v = L \otimes_k k_v$ and 
$$
\Gamma_v = \left\{(\delta,n) \in L_v^*\times k_v^*\,\,|\,\, N(\delta) = n^p\right\}.
$$
We also assume that for each place $v$ of $k$, the curve $C$ has a
$k_v$-rational divisor class of degree~$1$.
As mentioned in \cite[Section 13]{PS:1997}, this assumption is automatically satisfied when the genus
\hbox{$g(C) = (d-2)(p-1)/2$} satisfies $g(C) \not \equiv 1 \pmod p$.  It implies that the injection $\Pic^0 C \to J(k)$ is an isomorphism (see \cite[Prop.~3.2 and~3.3]{PS:1997}). As before, we will abbreviate the product over all places of $k$ to $\prod_v$.
The bottom face of diagram (\ref{together}) then yields the front face of the following diagram, where, as before, we have identified $\H^1(J[\phi])$ with $\H^1(M/\mu_p)$. 
\begin{equation}\label{selmers}
\xymatrix@!C=0.5in{
&\prod_v J(k_v)/\phi J(k_v) \ar@{^{(}->}[rr]_(0.5){(x-T, y)_v} \ar@{=}'[d][dd]
\ar@/^5.5mm/[rrrr]^{(x-T)_v}
&& \prod_v \Gamma_v/\chi(L_v^*)\iota(k_v^*) \ar@{^{(}->}'[d][dd]^(0.3){(\delta_\chi)_v} \ar[rr] && \prod_v L_v^*/{L_v^*}^pk_v^*\ar@{^{(}->}[dd]^{(\delta_p)_v} \\
J(k)/\phi J(k) \ar@{^{(}->}[rr]_(0.34){(x-T, y)} \ar@{=}[dd]\ar[ru]^r
&& \Gamma/\chi(L^*)\iota(k^*) \ar@{^{(}->}[dd]^(0.3){\delta_\chi} \ar[ru]^r\ar[rr] && L^*/{L^*}^pk^*\ar@{^{(}->}[dd]^(0.3){\delta_p} \ar[ru]^r\\
&\prod_v J(k_v)/\phi J(k_v) \ar@{^{(}->}'[r][rr]_(0.2){(\delta_\phi)_v}  && \prod_v \H^1(G_v, J[\phi]) \ar'[r][rr] && \prod_v \H^1(G_v,\mu_p(\Lsep)/\mu_p) \\
J(k)/\phi J(k) \ar@{^{(}->}[rr]_(0.5){\delta_\phi}  \ar[ru]^r&& \H^1(J[\phi]) \ar[rr] \ar[ru]^r&& \H^1(\mu_p(\Lsep)/\mu_p) \ar[ru]^r\\
}
\end{equation}
For each map 
in this front face, there is an analogous map over each completion $k_v$ of $k$. Taking the product over all places gives the back face of the diagram, while $r$ denotes each map from a global group to the product of the analogous local groups. 

The image of $J(k)/\phi J(k)$ in each of the four global groups is contained in the inverse image under $r$ of the image of 
$\prod_v J(k_v)/\phi J(k_v)$ in the corresponding product of local groups. We give three of these inverse images a name. 
\begin{align*}
\Sel^\phi (J,k) &= r^{-1}\left( \im \left( (\delta_\phi)_v \colon \prod_v J(k_v)/\phi J(k_v) \to \prod_v \H^1(G_v, J[\phi]) \right) \right),\\
\Sel^\phi_\fake (J,k) &= r^{-1}\left( \im \left( (x-T)_v\colon \prod_v J(k_v)/\phi J(k_v) \to \prod_v L_v^*/{L_v^*}^pk_v^* \right) \right),\\
\Sel^\phi_\unfake (J,k) &= r^{-1}\left( \im \left( (x-T,y)_v \colon \prod_v J(k_v)/\phi J(k_v) \to \prod_v  \Gamma_v/\chi(L_v^*)\iota(k_v^*) \right) \right).
\end{align*}

The {\em Selmer group} $\Sel^\phi (J,k)$ is commonly known. The {\em fake Selmer group} $\Sel^\phi_\fake (J,k)$ was introduced by Poonen and Schaefer in \cite{PS:1997}. The two groups are related by an exact sequence
$$
\mu_p \to \Sel^\phi (J,k) \to \Sel^\phi_\fake (J,k) \to 0,
$$
and it is also known when the first map is injective (see \cite[Thm.~13.2]{PS:1997}).
However, it is not always obvious whether the image 
of $J(k)/\phi J(k)$ in $ \Sel^\phi (J,k)$ maps injectively to $ \Sel^\phi_\fake (J,k)$. This means that although the fake Selmer group 
is more practical to work with explicitly, in doing so information may be lost. The following theorem shows that no information is 
lost when we work instead with the {\em explicit Selmer group} $\Sel^\phi_\unfake (J,k)$, which is just as easy to work with as the
fake Selmer group. 

\begin{thm}\label{isoms}
The map $\delta_\chi$ induces an isomorphism
$ \Sel^\phi_\unfake (J,k) \to \Sel^\phi (J,k) $.
\end{thm}
\begin{proof}
The fact that $\delta_\chi$ maps $ \Sel^\phi_\unfake (J,k)$ injectively to $\Sel^\phi (J,k)$
is clear, so it remains to prove surjectivity. Note that we have an isomorphism 
$\H^2(\mu_p) \isom \Br(k)[p]$. Therefore, identifying  $\H^1(J[\phi])$ with $\H^1(M/\mu_p)$ through $\epsilon_*$ as before,
the long exact sequences associated to the vertical short exact sequences in diagram (\ref{defM}), together with the results of Proposition~\ref{diagramone}, give rise to a commutative diagram with exact columns:
\[ \xymatrix{ \Gamma/\chi(L^*)\iota(k^*) \ar[d]^{\delta_\chi} \ar[r]
                    & L^*/{L^*}^p k^* \ar[d]^{\delta_p} \\
              H^1(J[\phi]) \ar[d]^{\delta_1} \ar[r]
                    & H^1(\mu_p(L\sep)/\mu_p) \ar[d]^{\delta_2} \\
              \Br(k)[p] \ar@{=}[r] & \Br(k)[p]
            }
\]
An analogous statement holds for every completion $k_v$ of $k$. 
Now suppose we have an element $\xi \in \Sel^\phi (J,k)$.  Then by definition $r(\xi)$ is contained in the image of $(\delta_\phi)_v$ and therefore in 
the image of $(\delta_\chi)_v$ (see diagram (\ref{selmers})). It follows that $r(\xi)$ maps to $0$ in $\prod_v \Br(k_v)[p]$ under the product of the local versions of $\delta_1$. Since the map 
$\Br (k)[p] \to \prod_v \Br(k_v)[p]$ is injective, we conclude $\delta_1(\xi)=0$, so there is an element $\eta \in \Gamma/\chi(L^*)\iota(k^*)$ with 
$\delta_\chi(\eta) = \xi$. A short diagram chase shows $\eta \in \Sel^\phi_\unfake (J,k) $, so $\delta_\chi\colon \Sel^\phi_\unfake (J,k) \to \Sel^\phi (J,k) $ 
is indeed surjective. 
\end{proof}

\begin{rem}
Similarly, the map $\delta_p$ 
induces an isomorphism from $\Sel^\phi_\fake (J,k)$ to the group
$$
r^{-1}\left( \im \left( \prod_v J(k_v)/\phi J(k_v) \to \prod_v \H^1(G_v,\mu_p(\Lsep)/\mu_p) \right) \right).
$$
\end{rem}

\begin{proof}[Proof of Theorem \ref{mainthm}]
The map $(x-T,y) \colon J(k) \to \Gamma/\chi(L^*)\iota(k^*)$ factors as 
$$
J(k) \to J(k) / \phi J(k) \to \Sel^\phi_\unfake(J,k) \subset  \Gamma/\chi(L^*)\iota(k^*).
$$
Theorem \ref{mainthm} therefore follows immediately from Theorem \ref{isoms}.
\end{proof}

\end{document}